\documentclass{amsart}
\usepackage{amssymb,latexsym}
\usepackage{amsmath}
\usepackage{graphicx}
 \usepackage{color}
\usepackage{enumerate}
\usepackage{url}
\newenvironment{enumeratei}{\begin{enumerate}[\upshape (i)]}{\end{enumerate}}
\numberwithin{equation}{section}
\theoremstyle{plain}
 \newtheorem{theorem}{Theorem}[section]
 \newtheorem{lemma}[theorem]{Lemma}
 \newtheorem{proposition}[theorem]{Proposition}

\theoremstyle{definition}

 \newtheorem{remark}[theorem]{Remark}

\newcommand \datum{\red{October 24, 2017,\quad 17:43}}

\newcommand \what [1] {\widehat #1}
\newcommand \nonca {N^{\textup{an}}}
\newcommand \noncn {N^\textup{num}}

\newcommand \qqp {\mathbb Q_{(p)}}
\newcommand  \pDb {D^{[\textup{o,}k]}}
\newcommand  \pfb {f^{[\textup{o,}k]}}
\newcommand  \pDa {D^{[\textup{e,}k]}}
\newcommand  \pfa {f^{[\textup{e,}k]}}

\newcommand  \ctwone {g_{\textup{cos}}^{(1)}}

\newcommand  \csg {g_{\textup{cos}}^{(k)}}

\newcommand \sfan [2] {F_{\textup{sl}}^{#1}(#2)}
\newcommand \dfan [2] {F_{\textup{cd}}^{#1}(#2)}

\newcommand \alg [1] {\mathbf{#1}}

\newcommand \zz {\mathbb Z}

\newcommand \qq {\mathbb Q}
\newcommand \rr {\mathbb R}

\newcommand \degx [2] {\textup{deg}_{#1}(#2)}
\newcommand \notmid {\mathrel{
\not\mathord{\kern -0.5pt \vert}}}
\newcommand \snotmid {\mathrel{
\not\mathord{\kern 2.2pt \vert}}}
\newcommand \binomial [2] {{{#1}\choose {#2}}}

\newcommand \tuple [1] {\langle #1\rangle}
\newcommand \pair [2] {\tuple{#1,#2}}
\renewcommand\phi{\varphi}

\newcommand\red[1]{{\textcolor{red}{#1}}}
\newcommand \tbf [1] {\textbf{#1}} 
\newcommand \set[1] {\{#1\}}

\renewcommand\phi{\varphi}
\renewcommand\epsilon{{\varepsilon}}

\newcommand\nothing [1] {}

\newcommand \fun[1]{{^\star\kern-2pt{#1}} }
\newcommand \ffun[1]{{^\star\kern-1pt{#1}} } 
\newcommand \varfun[1]{{^\star\kern-1pt{(#1)}} }
\newcommand \bvarfun[1]{{^\star\kern-1pt{\bigl(#1\bigr)}} }

\long\def\nothing #1  {}
\long\def\needsoon #1  {}
\begin{document}
\title[Geometric constructibility of polygons]
{Geometric constructibility of
polygons lying on a circular arc}

\author[D. Ahmed]{Delbrin Ahmed}
\email{delbrin@math.u-szeged.hu; delbrinhahmed@hotmail.co.uk}
\address{University of Szeged, Bolyai Institute\, and the University of Duhok}

\author[G.\ Cz\'edli]{G\'abor Cz\'edli}
\email{czedli@math.u-szeged.hu}
\urladdr{http://www.math.u-szeged.hu/\textasciitilde{}czedli/}

\address{University of Szeged, Bolyai Institute. 
Szeged, Aradi v\'ertan\'uk tere 1, HUNGARY 6720}
\author[E.\ K.\ Horv\'ath]{Eszter K.\ Horv\'ath}
\email{horeszt@math.u-szeged.hu}\urladdr{http://www.math.u-szeged.hu/\textasciitilde{}horvath/}\address{University of Szeged, Bolyai Institute. 
Szeged, Aradi v\'ertan\'uk tere 1, HUNGARY 6720}

\thanks{This research of the second and third authors is
supported by
the Hungarian Research Grant KH 126581}



\date{\datum\kern3cm\\
\phantom{nk} 2000 \emph{Mathematics Subject Classification}. Primary 51M04, secondary 12D05}

\begin{abstract} For a positive integer $n$, an \emph{$n$-sided polygon lying on a circular arc} or, shortly, an \emph{$n$-fan} is a sequence of $n+1$ points on a circle going counterclockwise such that the ``total rotation'' $\delta$ from the first point to the last one is at most $2\pi$. We prove that for $n\geq 3$, the $n$-fan cannot be constructed with straightedge and compass in general from its \emph{central angle} $\delta$ and its
\emph{central distances}, which are 
the distances of the edges from the center of the circle. 
Also, we prove that for each \emph{fixed} $\delta$ in the interval $(0, 2\pi]$ and
for every $n\geq 5$, there exists a \emph{concrete} $n$-fan with central angle $\delta$ that
is not constructible from its central distances and $\delta$.
The present paper generalizes some earlier results  published by the second author and \'A.~Kunos 
on the particular cases $\delta=2\pi$ and $\delta=\pi$.
\end{abstract}


\maketitle

\section{Introduction and our results}\label{sectionintro}
\subsection*{A short historical survey}
With the exception of squaring the circle, not much research interest was paid to geometric constructibility problems
for one and a half centuries after the 
Gauss--Wantzel Theorem in  \cite{wantzel}, which  completely described the constructible regular $n$-gons.  
This can be well explained by the fact that 
most of the ancient constructibility problems as well as constructing triangles from various given data are too elementary and, furthermore, nowadays it does not require too much skill to solve them with the help of computer algebra in few minutes. This is exemplified by the textbooks Cz\'edli~\cite{czgproblembook} and Cz\'edli and Szendrei~\cite{czgsza}, where more than a hundred constructibility problems are solved. 

It was Schreiber~\cite{schreiber} who revitalized the research of geometric constructibility by an interesting non-trivial problem, the constructibility of cyclic (also known as inscribed) polygons from their side lengths. Furthermore, he pointed out that this problem requires a variety of interesting tools from algebra and geometry. The first complete proof of his theorem on the non-constructibility of cyclic  $n$-gons from their side lengths for every $n\geq 5$ used some involved tools even from analysis; see Cz\'edli and Kunos~\cite{czgkunos}.

\begin{figure}[htb]
\centerline
{\includegraphics[scale=0.9]{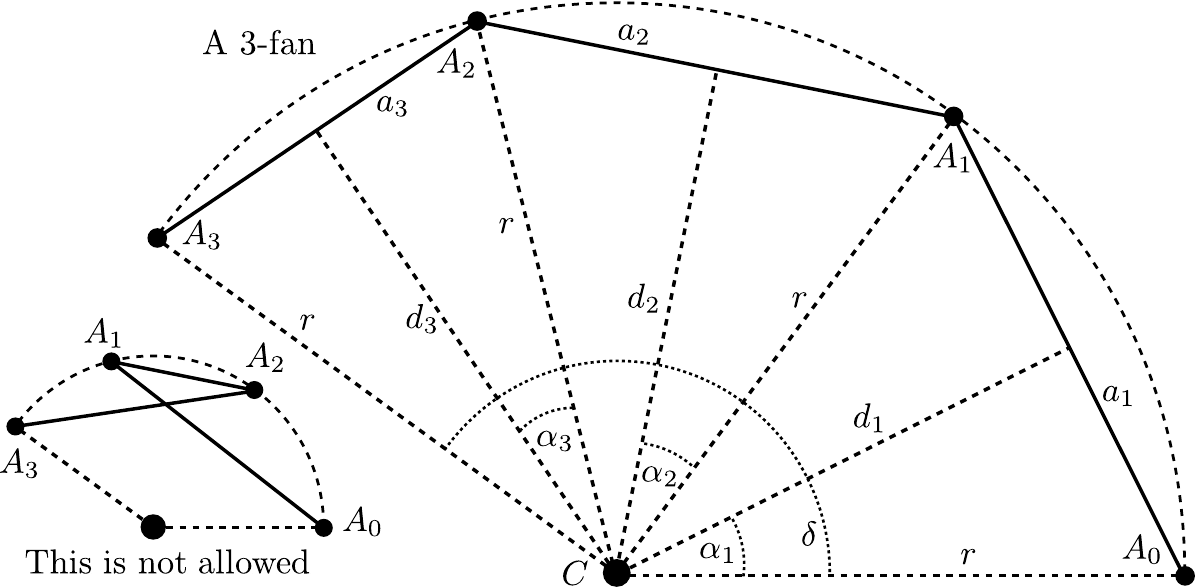}}
\caption{A $3$-fan, that is, a 3-sided polygon lying on a circular arc,  and a ``zigzag polygon'', which is not investigated in the paper
\label{fig1}}
\end{figure}

\subsection*{Polygons on a circular arc}
Let $n\in\mathbb N=\set{1,2,3,\dots}$.
By an \emph{$n$-sided polygon lying on a circular arc} or, shortly, by an \emph{$n$-fan} we mean a planar polygon $\alg A=\tuple{A_0,A_1,\dots,A_n}\in(\mathbb R^2)^{n+1}$ such that the \emph{vertices} $A_0,\dots, A_n$, in this order, lie on the same circular arc, see on the right of Figure~\ref{fig1}. The short name ``$n$-fan'' is explained by the similarity with a not fully open hand fan. 
Some important real numbers that determine an $n$-fan are also given in Figure~\ref{fig1}; the  \emph{central distances} $d_1,\dots, d_n$ of the sides from the center $C$ of the circular arc, the \emph{central angle} $\delta\in(0,2\pi]=\set{r\in \rr: 0<r\leq 2\pi}$, and the \emph{radius} $r$ are worth separate mentioning here.
Like in the earlier papers Schreiber~\cite{schreiber}, Cz\'edli and Kunos~\cite{czgkunos}, and Cz\'edli~\cite{czgthalesian}, an easy argument based on properties of continuous real functions shows that the ordered tuple  $\tuple{\delta; d_1,\dots,d_n}$ determines the $n$-fan up to isometry, provided that $0<\delta<2\pi$ or $n\geq 3$.  We denote by $\dfan n{\delta; d_1,\dots,d_n}$ the $n$-fan determined by this tuple; the subscript comes from ``central distances''. For the central angle $\delta$, we always assume that $0<\delta\leq 2\pi$. Furthermore, we always assume that our $n$-fans are \emph{convex} in the sense that the angle $\angle(A_{i-1}A_iA_{i+1})$ at $A_i$ contains $C$ for 
$i\in\set{1,\dots,n-1}$. If  $\delta=2\pi$, then we assume also that the angle $\angle(A_{n-1}A_0A_{1})$ at $A_0=A_n$ contains $C$. Convexity means that 
``zigzag polygons'' like the small one on the left of Figure~\ref{fig1} are not allowed. 
With the notation  $\mathbb R^+=\set{x\in\mathbb R: x>0}$, note that there are $(n+1)$-tuples
$\tuple{\delta,d_1,\dots,d_n}\in (0,2\pi]\times (\mathbb R^+)^n$ for which $\dfan n{\delta,d_1,\dots,d_n}$ does not exist. However, similarly to Cz\'edli and Kunos~\cite{czgkunos} and Cz\'edli~\cite{czgthalesian}, it follows from continuity that, for $n\in\mathbb N$, 
\begin{equation}
\parbox{9.4cm}{
if $n\geq 3$ or $\delta<2\pi$, and  all the ratios $d_i/d_j$ are sufficiently close to 1, then $\dfan{n}{\delta,d_1,\dots,d_n}$ exists and it is unique.}
\label{parboxscGjmB}
\end{equation}

\subsection*{Constructibility}
In this paper, \emph{constructibility} is always
understood as the classical geometric constructibility with straightedge and compass.
(We prefer the word ``straightedge'' to ``ruler'', because it describes the permitted usage better.)
Due to the usual coordinate system of the plane, we can assume that a \emph{constructibility problem} is always a task of constructing a real number $t$ from a sequence $\tuple{t_1,\dots, t_m}$ of real numbers. Geometrically, this means that we are given the points $\pair 0 0$, $\pair {t_1}0$, \dots, $\pair {t_m}0$ in the plane and we want to construct the point $\pair t 0$. Angles are also given by real numbers.
Whenever we say that the central angle $\delta$ is given, this means that the real number  $p:=\cos(\delta/2)$ is given. From the perspective of constructibility, any other usual way of giving $\delta$ is equivalent to giving $p$, that is the point $\pair p0$. The advantage of using $p=\cos(\delta/2)\in[-1,1)$ over, say, $\cos\delta$ is that $p$ uniquely determines $\delta\in (0,2\pi]$.
As opposed to the constructibility of a concrete point from other concrete points, the concept of \emph{constructibility in general} is more involved; the reader may want to but need not see Cz\'edli~\cite{czgthalesian} and Cz\'edli and Kunos~\cite{czgkunos}  for a rigorous definition. The reader of \emph{this} paper may safely assume that ``constructible in general'' means ``constructible for all meaningful data''.

\subsection*{Our results}
Our first target is to decide whether the $n$-fan $\dfan{n}{\delta,d_1,\dots,d_n}$ 
can be constructed from $\tuple{\delta,d_1,\dots,d_n}$ 
in general. 
We are going to prove the following.

\begin{theorem}\label{thmmain}
The $n$-fan  $\dfan{n}{\delta,d_1,\dots,d_n}$ is geometrically constructible in general from $\tuple{\cos(\delta/2),d_1,\dots,d_n}$ if and only if $n\in\set{1,2}$. Furthermore, if $n\geq 3$, then there exist \emph{rational} numbers $p=\cos(\delta/2)$, $d_1$, \dots, $d_n$ such that 
 $|\set{d_1,\dots,d_n}|\leq 2$ holds and the $n$-fan 
$\dfan{n}{\delta,d_1,\dots,d_n}$ exists, but this $n$-fan cannot be  constructed from $\tuple{p,d_1,\dots,d_n}$.
\end{theorem}

For many values of $n$, the  inequality $|\set{d_1,\dots,d_n}|\leq 2$ above can easily be strengthened to   the equality $|\set{d_1,\dots,d_n}|=1$. For example, for $n=3$ and $\delta=2\pi/3$, where $p=\cos(\delta/2)=1/2$, even the $3$-fan   $\dfan3{2\pi/3,1,1,1}$ cannot be constructed; this follows trivially from the Gauss--Wantzel Theorem, \cite{wantzel}, from which we know that the regular  nonagon (also known as $9$-gon) cannot be constructed. Note that this easy argument is not applicable if $n$ is a power of 2.
Note also that the constructibility from \emph{rational} parameters is equivalent to the constructibility from $\tuple{1}$, that is, from the points $\pair 0 0$ and $\pair 1 0$.

In view of earlier results where $\delta$ was fixed as $2\pi$ or $\pi$, it is reasonable to consider the problem also for the case where $\delta$ is \emph{fixed} and only the side lengths $d_1,\dots,d_n$ are ``general''.  In particular, if $\delta$ is a constructible angle like $\pi$, $\pi/3$ or $\pi/2$, then we can consider it only information rather than a part of the data. As a preparation for our second theorem, we introduce the following notation for $\delta\in (0,2\pi]$ and $m\in \mathbb N$:
\begin{align*}
\noncn(\delta)&:=\{n\in \mathbb N: \text{there exist }d_1,\dots,d_n\in \rr^+\text{ such that }\dfan{n}{\delta,d_1,\dots,d_n}\cr
&\phantom{xxxi}\text{also exists and it is uniquely determined, but it is}\cr
&\phantom{xxxi}\text{not constructible from }\tuple{\cos(\delta/2),d_1,\dots,d_n}\},\quad \text{ and}\cr
\nonca(m)&:=\{\delta\in (0,2\pi]: m\in \noncn(\delta)\}.
\end{align*}
The superscripts above come for ``numbers'' and ``angles'', respectively, while ``$N$'' comes from the prefix ``non'' in ``non-constructible''.
As usual, $(0,2\pi)$ stands for the open interval $\set{r\in\rr: 0<r<2\pi}$ of real numbers.
Now, we are in the position to formulate our second statement.

\begin{theorem}\label{thmvltcorolyclS} The following six assertions hold. 
\begin{enumeratei}
\item\label{thmvltcorolyclSa} $\noncn(2\pi)=\set{3,5,6,7,8,\dots}$. In particular, $2\pi\in\nonca(3)$. 
\item\label{thmvltcorolyclSb} $\noncn(\pi)=\set{4,5,6,7,\dots}$. In particular, $\pi\notin\nonca(3)$. 
\item\label{thmvltcorolyclSc} For every $\delta\in(0,2\pi]$ such that $\cos(\delta/2)$ is transcendental, we have that $\noncn(\delta)=\set{3,4,5,6,\dots}$ and, in particular, $\delta\in \nonca(3)$.
\item\label{thmvltcorolyclSd} For every $\delta\in (0,2\pi)$,  $\,\,\set{4,5,6,\dots}\subseteq\noncn(\delta)\subseteq\set{3,4,5,6,\dots}$.
\item\label{thmvltcorolyclSe} For $5\leq n\in\mathbb N$, $\nonca(n)=(0,2\pi]$ but $\nonca(4)=(0,2\pi)$. 
\item\label{thmvltcorolyclSf}For $k,m\in \mathbb N$, let $\displaystyle{A_k^{(m)}=\bigl\{\frac {i^{1/m}}j: 1\leq j \leq k\text{ and } 1\leq i< j^m\bigr\}}$. 
With this notation, whenever $|\kern-0.5pt{\cos(\delta/2)}|$ belongs to $A_{1000}^{(1)}\cup A_{100}^{(2)}$, then $\delta\in\nonca(3)$.
\end{enumeratei}
\end{theorem}

\begin{remark}\label{remarkdhgzVbMq}
Note that if $\delta\in(0,2\pi)$ and $n\in \mathbb N\setminus\noncn(\delta)$, then for all $n$-tuples $\tuple{d_1,\dots,d_n}\in(\rr^+)^n$ such that
the $n$-fan $\dfan{n}{\delta,d_1,\dots,d_n}$  exists, this $n$-fan is constructible from $\tuple{\cos(\delta/2),d_1,\dots,d_n}$. However, this is not true for $\delta=2\pi$ since  $\dfan1{2\pi,d_1}$ and  $\dfan2{2\pi,d_1,d_2}$ do not make sense; the first is not determined uniquely while the second does not exist. 
\end{remark}

It is a surprising gap in \ref{thmvltcorolyclS}\eqref{thmvltcorolyclSa} that   $4$ does not belong to $\noncn(2\pi)$. The redundancy in the theorem focuses our attention to $\nonca(3)$. However,  we do not have a satisfactory description of $\nonca(3)$.
Note  that it follows from \ref{thmvltcorolyclS}\eqref{thmvltcorolyclSf} that 
\[\bigl\{\frac{j\pi}4: 4\neq j\in\set{1,2,\dots,7} \bigr\} \cup \bigl\{\frac{k\pi}6: 6\neq k\in\set{1,2,\dots,11} \bigr\} \subseteq \nonca(3).
\]
We could let $j$ and $k$ run up to 8 and 12, respectively, but the inclusion above for $j=8$ and $k=12$ follows from \ref{thmvltcorolyclS}\eqref{thmvltcorolyclSa}, not from \ref{thmvltcorolyclS}\eqref{thmvltcorolyclSf}.

The $n$-fan determined by its central angle $\delta$ and its \emph{side lengths} $a_1,\dots,a_n$, see Figure~\ref{fig1}, will be denoted by $\sfan{n}{\delta,a_1,\dots, a_n}$; the subscript comes from ``side lengths''. Due to the Limit Theorem from Cz\'edli and Kunos~\cite{czgkunos}, the constructibility problem for $\sfan{n}{\delta,a_1,\dots, a_n}$ is easier than that for  $\dfan{n}{\delta,d_1,\dots, d_n}$. This fact and space considerations explain that the present paper contains only the following result on side lengths.

\begin{proposition}\label{propsLengths} For $n\in\mathbb N$, the $n$-fan 
$\sfan{n}{\delta,a_1,\dots, a_n}$ is constructible in general from $\tuple{\delta,a_1,\dots, a_n}\in (0,2\pi)\times(\rr^+)^n$  if and only if $n\leq 2$.
\end{proposition}

\begin{remark}\label{remarkZhF}
For a fixed $\delta$, the situation can be different. We know from school and Cz\'edli and Kunos~\cite{czgkunos}  or  Screiber~\cite{schreiber}  that $\sfan{3}{2\pi,a_1,a_2,a_3}$ and $\sfan{4}{2\pi,a_1,\dots, a_4}$ can be constructed from $\tuple{a_1,a_2,a_3}$ and  $\tuple{a_1,\dots,a_4}$, respectively, in general. On the other hand, we know from Cz\'edli~\cite[Theorem 1.1(v)]{czgthalesian} that  $\sfan{3}{\pi,1,2,3}$ exists but it cannot be constructed from its side lengths.
\end{remark}

\subsection*{Outline} The rest of the paper is devoted to the proofs of our theorems and also to some additional statements that make these theorems a bit stronger by tailoring special conditions on possible data determining non-constructible $n$-fans. Section~\ref{sectionsurvey} lists some well-known concepts, notations, and facts from algebra for later reference; readers familiar with irreducible polynomials and field extensions may skip most parts of this section. Section~\ref{sectionproofprop} contains the above-mentioned additional statements as propositions, and it contains almost all the proofs of the paper.

\section{A short overview of the algebraic background}\label{sectionsurvey}
A polynomial is \emph{primitive} if the greatest common divisor  of its coefficients is 1.
The following well-known statement is due to C.\,F.\ Gauss; we cite parts \eqref{lemmagaussa} and \eqref{lemmagaussc}  from  Cameron~\cite[Theorem 2.16 (page 90) and Proposition 7.24 (page 260)]{cameron}, while  \eqref{lemmagaussb}  follows from \eqref{lemmagaussc}.

\begin{lemma}\label{lemmagauss} If $R$ is a unique factorization domain with field of fractions $F$, then
\begin{enumeratei}
\item\label{lemmagaussa} the polynomial ring $R[x]$ is also a unique factorization domain, 
\item\label{lemmagaussb} if a polynomial is irreducible in $R[x]$, then it is also irreducible in $F[x]$, and
\item\label{lemmagaussc} a primitive polynomial is irreducible in $R[x]$ iff it is irreducible in $F[x]$.
\end{enumeratei}
\end{lemma}
For the ring $\zz$ of integers and $k\in\mathbb N$, the field of fractions of $\zz[x_1,\dots,x_k]$ is $\qq(x_1,\dots,x_k)$, the field of rational $k$-variable functions over $\qq$. Note that for $c_1,\dots,c_k\in\rr$, we say that these  numbers are \emph{algebraically independent} over $\qq$ if the map $f(x_1,\dots,x_k)\mapsto f(c_1,\dots,c_k)$ from $\zz[x_1,\dots,x_k]$  to $\rr$ extends to a field embedding  $\qq(x_1,\dots,x_k)$ to $\rr$. For $k=1$, this means that $c_1$ is a transcendental number (over $\qq$). The field generated by $\qq{}\cup\set{c_1,\dots,c_n}$ is denoted by $\qq(c_1,\dots,c_k)$; it is isomorphic to $\qq(x_1,\dots,x_k)$ provided that 
$c_1,\dots,c_k\in\rr$  are algebraically independent over $\qq$ .
We often write $\qqp$ instead of $\qq(p)$, even if $p$ is not transcendental.

Given a unique factorization domain $R$ with field of fractions $F$, the polynomial rings $R[x,y]$, $R[x][y]$, and $R[y][x]$ are well known to be isomorphic. This fact allows us to write $f_x(y)$ and $f_y(x)$ instead of $f(x,y)\in R[x,y]$. That is, $f_x(y)$, $f_y(x)$, and  $f(x,y)$ are essentially the same polynomials but we put an emphasis on $f_x(y)\in R[x][y]\subseteq F(x)[y]$ and  $f_y(x)\in R[y][x]\subseteq F(y)[x]$. Therefore, 
the following convention applies in the paper: 
\begin{equation}
\parbox{9cm}{no matter which of $f(x,y)\in R[x,y]$, $f_x(y)\in R[x][y]$, and $f_y(x)\in R[y][x]$ is given first, we can also use the other two.}
\label{teqconV}
\end{equation}
Note that in many cases but not always, $R$ and $F$ will be $\zz$ and $\qq$. 
 The \emph{degree} of a polynomial $g(x)$ will be denoted by $\degx x{g(x)}$ or $\degx x{g}$.

The following statement is well known and usually taught for MSc students;
see, for example, Cz\'edli and Szendrei~\cite[Theorem V.3.6]{czgsza}; see also the list of references right before Cz\'edli and Kunos~\cite[Proposition 3.1]{czgkunos}.

\begin{proposition}\label{propwhncStr}
Let $u,c_1,\dots,c_t\in \rr$. If there exists an irreducible polynomial 
\[h(x)\in\qq(c_1,\dots,c_t)[x]\] 
such that $h(u)=0$ and  $\degx x{h(x)}$ is not a power of $\,2$, then $u$ is not constructible from $\qq\cup \set{c_1,\dots,c_t}$ \textup{(\kern-0.5pt}or, equivalently and according to the present terminology, $u$ is not constructible from $\tuple{1,c_1,\dots,c_t})$.
\end{proposition}

The following statement is also well known, and it is even trivial for fields rather than unique factorization domains; having no reference at hand, we are going to give a proof.

\begin{lemma}\label{lemmaquadrpol}
Let $R$ be a unique factorization domain with field of fractions $F$. Let $f(x)=ax^2+bx+c\in R[x]$ be a \emph{primitive} quadratic polynomial. If its discriminant, $D=b^2-4ac$, is not a square in $R$, then $f(x)$ is irreducible in $R[x]$ and, consequently, also in $F[x]$.
\end{lemma}

\begin{proof} Suppose to the contrary that $f(x)$ is reducible. Since it is primitive, it cannot have a nontrivial divisor of degree 0. Hence, there are $a_1,b_1,a_2,b_2\in R$ such that $ax^2+bx+c=f(x)=(a_1x+b_1)(a_2x+b_2)$. Comparing the leading coefficients, $a=a_1a_2$. Since $-b_1/a_1$ is a root of $f(x)$, the well-known formula gives that
\[ -\frac{b_1}{a_1}=\frac{-b\pm\sqrt D}{2a},
\]
After multiplying by $2a=2a_1a_2$, we obtain that 
$-2b_1a_2= -b\pm\sqrt D$.
Therefore, $D=(b-2b_1a_2)^2$ is a square of $b-2b_1a_2\in R$. This contradicts our hypothesis and proves the lemma.
\end{proof}

\section{Proofs and propositions}\label{sectionproofprop}

\begin{proposition}\label{propevennddd}
If $n\geq 4$ is an \emph{even} integer, then 
for every real number $\delta$ belonging to the open interval $(0,2\pi)$, there 
exists a \emph{rational} number $c$ such that the $n$-fan
$\dfan{n}{\delta,1,\dots,1,c}$ exists but it cannot be constructed from $\tuple{\cos(\delta/2),1,c}$.
\end{proposition}

\begin{proof} The case $\delta=\pi$ has been settled in Cz\'edli~\cite{czgthalesian}; see Cases 3 and 4 in page 68 there and note that our $n$ corresponds to $n+1$ in \cite{czgthalesian} and $\sqrt c$ and $c$ are equivalent data from the perspective of geometric constructibility. Hence, we can assume that $\delta\neq \pi$. 
We denote $\cos(\delta/2)$ by $p$; it belongs to the open interval $(-1,1)$ and it is distinct from 0. The smallest subfield of $\rr$ that includes  $\qq\cup\set p$ is denoted by $\qqp$.
We now from \eqref{parboxscGjmB} that if $c$ is sufficiently close to 1, then $\dfan{n}{\delta,1,\dots,1,c}$ exists. This fact and the Rational Parameter Theorem of Cz\'edli and Kunos~\cite[Theorem 11.1]{czgkunos} yield that it suffices to show that 
$\dfan{n}{\delta,1,\dots,1,c}$ is not constructible for 
those $c$ in a small neighborhood of 1 that are transcendental over $\qqp$. Since $\qqp(c)$ is isomorphic to the field $\qqp(y)$ of rational functions over $\qqp$ for these transcendental $c$, we can treat $c$ later as an indeterminate $y$. Note that this paragraph, that is the first paragraph of the present proof, would also be appropriate for $\dfan{n}{\delta,1,\dots,1,c,c}$; this fact will be needed only in another proof of the paper.

Let $k:=n-1$; it is an odd number and $k\geq 3$. As always in this paper, $r$ denotes the radius of the circular arc.  We let $u:=1/r$. As Figure~\ref{fig1} approximately shows, for the ``half angles''
$\alpha:=\alpha_1=\dots=\alpha_{k}$ and $\beta:=\alpha_n$, we have that 
\begin{equation}
\cos\alpha=u,\quad \sin\alpha=\sqrt{1-u^2}, 
\quad \cos\beta=c u,\quad \sin\beta= \sqrt{1-c^2u^2}\text.
\label{eqcosiperr}
\end{equation} 
Since we work with half angles, $k\alpha+\beta=\delta/2$, whereby $k\alpha=\delta/2 - \beta$. 
Using the well-known formula for the cosine of a difference, we obtain that 
\begin{equation}
\begin{aligned}
\cos(\delta/2 - \beta)&=\cos(\delta/2)\cos\beta + \sin(\delta/2)\sin\beta \cr
&= p c u + \sqrt{1-p^2}\cdot \sqrt{1-c^2u^2}.
\end{aligned}
\label{eqdzghmnGrT}
\end{equation}
We also need the following well-known equality, which we combine with \eqref{eqcosiperr}:
\begin{equation}
\begin{aligned}
\cos(k\alpha)&=\sum_{2\mid j=0}^k (-1)^{j/2} \binomial k j ( \cos \alpha)^{k-j}\cdot (\sin\alpha)^{j}\cr
&= \sum_{2\mid j=0}^k (-1)^{j/2} \binomial k j u^{k-j}\cdot (1-u^2)^{j/2}=:\csg (u).
\end{aligned}
\label{cossum}
\end{equation}
Note that $\csg $ is a polynomial over $\zz$ since $j$ above runs through \emph{even} numbers. 
Since the coefficient of $u^k$ is
\begin{equation}
\sum_{2\mid j=0}^k (-1)^{j/2} \binomial k j ({-1})^{j/2} =  \sum_{2\mid j=0}^k  \binomial k j=2^{k-1}\neq 0,
\label{eqtwkmoeven}
\end{equation}
we conclude that 
\begin{equation}
\parbox{7.2cm}{the leading coefficient of $\csg(u)$ is a positive integer and the degree of $u$ in $\csg(u)$ is $k$.}
\label{pboxfokg}
\end{equation}
Since $k\alpha=\delta/2-\beta$,  \eqref{eqdzghmnGrT} and \eqref{cossum} give the same real number. After 
rearranging the equality of  \eqref{eqdzghmnGrT} and \eqref{cossum} and squaring, 
\begin{equation}
\bigl(\csg(u)-p c u\bigr)^2 = (1-p^2)(1-c^2u^2).
\label{eqdwsiezBfhrB}
\end{equation}
This encourages us to consider the polynomial
\begin{equation}
\pfa (x,y)= \bigl(\csg(x)-p y x\bigr)^2 - (1-p^2)(1-y^2x^2)\in \qqp[x,y],
\label{eqfxzegz}
\end{equation}
which is obtained from \eqref{eqdwsiezBfhrB} by  substituting
$\pair u c  \leftarrow \pair x y$ and rearranging. The superscript of $\pfa$ reminds us to ``even'' and $k$. Since $k\geq 3$ and it is odd,
the  degree $\deg_x(\pfa )$ of $\pfa $ in $x$ is $2\cdot\deg_x (\csg)=2k$ by \eqref{pboxfokg}, whence $\deg_x(\pfa )$ is not a power of 2. Note that $\deg_x(\pfa )$ remains the same if we replace $y$ by $c$, since $c$ is transcendental over $\qqp$. Therefore, Proposition~\ref{propwhncStr} will imply the non-constructibility of $u$ and that of our polygon as soon as we show that
$\pfa (x,c)=\pfa_{c}(x)\in \qqp(c)[x]$ is an irreducible polynomial. Let $\phi\colon\qqp(c)\to \qqp(y)$ be the canonical isomorphism that acts identically on $\qqp$ and maps $c$ to $y$. 
This $\phi$ extends to an isomorphism $\what\phi\colon\qqp(c)[x]\to \qqp(y)[x]$ with the property $\what\phi(x)=x$ in the
usual way.
It suffices to show that $\what\phi\bigl(\pfa (x,c)\bigr)$ is irreducible in $\qqp(y)[x]$. But $\what\phi\bigl(\pfa (x,c)\bigr)=\pfa_y (x)\in \qqp[y][x]$ and $\qqp(y)$ is the field of fractions of $\qqp[y]$. Thus, by Lemma~\ref{lemmagauss}, it suffices to show that $\pfa_y(x)=\pfa(x,y)$ is irreducible in $\qqp[y][x]\cong \qqp[x,y]\cong \qqp[x][y]$. So, in the rest of the proof, we deal only with the irreducibility of 
the polynomial $\pfa_x(y)=\pfa(x,y)$.

Rearranging  \eqref{eqfxzegz} according to the powers of $y$, we obtain that
\begin{equation}
\begin{aligned}
\pfa_x(y)&= \Bigl(p^2x^2 + (1-p^2)x^2\Bigr)\cdot y^2
-2x p \csg(x)\cdot y  \cr
&\kern 4.2cm +\Bigl(\csg(x)^2 -(1-p^2) \Bigr) \cr
&= x^2\cdot y^2 -2 p x \csg(x)\cdot y  + (\csg(x)^2 +p^2-1) \in \qqp[x][y].
\end{aligned}
\label{eqdzmBfQ}
\end{equation}
Since $p\in(-1,1)\setminus\set 0$, we have that $-1<p^2-1<0$, whence  $p^2-1$ is not an integer. Thus, since $\csg(x)\in \zz[x]$, the 
constant term in $\csg(x)^2 +p^2-1$ is nonzero. 
In $\qqp[x]$, which is a unique factorization domain, $x$ is an irreducible element. The above-mentioned nonzero term guarantees that $x$ does not divide $\csg(x)^2 +p^2-1$. Thus, $\pfa_x(y)$ is a primitive polynomial over  $\qqp[x]$, and we are going to apply Lemma~\ref{lemmaquadrpol}. To do so, we compute the discriminant $\pDa_{x}$ of $\pfa_x(y)$ as follows:
\begin{equation}
\begin{aligned}
\pDa_{x} &:= 4p^2 x^2\csg(x)^2 - 4 x^2(\csg(x)^2+p^2-1)\cr
&\phantom{:}=4(p^2-1)x^2 \cdot \bigl(\csg(x)^2 -1\bigr).
\end{aligned}
\label{eqdcsRmnt}
\end{equation}
Since $p^2-1<0$, it follows from \eqref{pboxfokg} that
\begin{equation}
\pDa_{t}\text{ tends to }\mathord-\infty\,\,\text{ as }t\in \qqp\text{ tends to }\infty.
\label{eqtndstmnftY}
\end{equation}
Now if $\pDa_x$ was of the form $h(x)^2$ for some $h(x)\in \qqp[x]$, then we would have that $\pDa_{t}=h(t)^2\geq 0$ for all $t\in \qqp$ and \eqref{eqtndstmnftY} would be impossible. Hence, $\pDa_{x}$ is not a square in $\qqp[x]$ and Lemma~\ref{lemmaquadrpol} yields the irreducibility of $\pfa_x(y)$, as required.
This completes the proof of Proposition~\ref{propevennddd}.
\end{proof}

Remark~\ref{remarkdhgzVbMq}
explains why we consider $(0,2\pi)$ rather than $(0,2\pi]$ in the following statement.

\begin{proposition}\label{propositionTcDDgNgl} If $n\in\set{1,2}$, then 
for every real number $\delta\in (0,2\pi)$, the $n$-fan
$\dfan{n}{\delta,d_1,\dots,d_n}$ can be constructed from $\tuple{\cos(\delta/2), d_1,\dots, d_n}$ in general.
\end{proposition}

\begin{proof} We assume that $n=2$ since the case $n=1$ is trivial by  elementary geometrical considerations. We can also assume that the  scale is chosen so that $d_1=1$. Let $c=d_2$. It is clear by \eqref{eqcosiperr} and \eqref{cossum} that $\ctwone(u)=\cos(\alpha)=u$. Substituting this into \eqref{eqdwsiezBfhrB}, an easy calculation leads to 
\begin{equation}(c^2 -2p c +1 )u^2+p^2-1=0.
\label{eqdibvzhbSnBtPs}
\end{equation} 
Since $p\in(-1,1)$,  $p^2-1$ is distinct from $0$. Hence, \eqref{eqdibvzhbSnBtPs} gives that the coefficient $c^2 -2pc +1$ of $u^2$ is nonzero. Thus, $u$ is the root of a quadratic polynomial over the field $\qq(p,c)$, whereby it is constructible. So are $r=1/u$ and our 2-fan. 
\end{proof}

\begin{proposition}\label{propoveoDdDdd}
If $n\geq 5$ is an \emph{odd} integer, then 
for every real number $\delta$ belonging to the open interval $(0,2\pi)$, there 
exists a \emph{rational} number $c$ such that 
\begin{enumerate}[\phantom{xxx}\upshape(i)]
\item\label{propoveoDdDdda} if $\delta\neq \pi$, then the $n$-fan
$\dfan{n}{\delta,1,\dots,1,c,c}$ exists but it cannot be constructed from $\tuple{\cos(\delta/2),1,c,}$, and
\item\label{propoveoDdDddb} if $\delta =  \pi$, then the $n$-fan
$\dfan{n}{\delta,1,\dots,1,1,c}$ exists but it cannot be constructed from $\tuple{\cos(\delta/2),1,c}=\tuple{0,1,c}$.
\end{enumerate}
\end{proposition}

\begin{proof} 
First, we deal with \eqref{propoveoDdDdda}. Let $k=n-2$; note that $k$ is odd and $k\geq 3$. 
The first paragraph of the proof of Proposition~\ref{propevennddd} for $\dfan{n}{\delta,1,\dots,1,c,c}$ and  \eqref{eqcosiperr} will be used. In particular, $c$ is assumed to be transcendental, whence so is $c^2$. Since
\begin{equation}
\begin{aligned}
\cos(\delta/2 - 2\beta)&=\cos(\delta/2)\cos(2\beta) + \sin(\delta/2)\sin(2\beta) \cr 
&=p\bigl(2\cos^2(\beta)-1\bigr) + \sqrt{1-p^2}\cdot 2\sin\beta\cdot\cos\beta \cr
&= p (2 c^2 u^2 -1 ) + 2 c u\cdot  \sqrt{1-p^2}\cdot \sqrt{1-c^2u^2}
\end{aligned}
\label{eqdpzsMlTsQsT}
\end{equation}
and  $k\alpha+2\beta=\delta/2$ gives that 
$k\alpha=\delta/2 - 2\beta$,  \eqref{cossum} and \eqref{eqdpzsMlTsQsT} give the same value.  Rearranging the equality of these two values and squaring, we have that
\begin{equation}
\bigl(\csg(u) - p (2 c^2 u^2 -1 )\bigr)^2 =
 4 c^2 u^2 (1-p^2) (1-c^2u^2).
\label{eqdinbvmXwChg}
\end{equation}
Since $c$ and $c^2$ are mutually constructible from each other, we can assume that $c^2$ rather than $c$ is given. Rearranging \eqref{eqdinbvmXwChg} and substituting $\tuple {x, y}$ for $\tuple{ u, c^2}$, we obtain that $u$ is a root (in $x$) of the following polynomial
\begin{equation}
\begin{aligned}
\pfb_{y}(x) &= \pfb(x,y)=  \pfb_{x}(y)\cr
&= \bigl(\csg(x) - p (2 y x^2 -1 )\bigr)^2 -
 4 y x^2 (1-p^2) (1-y x^2)\cr
&=4x^4\cdot y^2 -\bigl(4p x^2\csg(x)+4x^2\bigr)\cdot y +\bigl( p+\csg(x)\bigr)^2.
\end{aligned}
\label{eqcjMJtzhBZnWrVt}
\end{equation}
Observe that $\deg_x(\pfb)=2\cdot\deg_x(\csg)=2k$ since $k\geq 3$. Thus, $\deg_x(\pfb)$  is not a power of $2$ since $k\geq 3$ is odd. Hence, by the same reason as in the paragraph right after \eqref{eqfxzegz}, it suffices to show that the quadratic polynomial $\pfb_{x}(y)$ is irreducible in $\qqp[x][y]$. 
The assumption  $\delta\in (0,2\pi)\setminus\set \pi$ gives that $0<p^2<1$.  Since $j$ is even in \eqref{cossum} but now $k$ is odd, the constant term in $\csg(x)$ is 0. So the constant term of $\bigl( p+\csg(x)\bigr)^2$ is $p^2\neq 0$. Hence, $x$, which is an irreducible element in the unique factorization domain $\qqp[x]$ and the only prime divisor of $4x^4$, does not divide $\bigl( p+\csg(x)\bigr)^2$. Thus, the quadratic polynomial $\pfb_{x}(y)$, see the last line of \eqref{eqcjMJtzhBZnWrVt}, is primitive. Its discriminant is 
\begin{equation}
\begin{aligned}
\pDb_{x}&= \bigl(4p x^2\csg(x)+4x^2\bigr)^2 
- 16x^4\cdot \bigl( p+\csg(x)\bigr)^2
\cr
&\phantom{:}= 16x^4(p^2-1)\cdot\bigl(\csg(x)^2-1\bigr) \in\qqp[x],
\end{aligned}
\label{eqdcsRsjhWms}
\end{equation}
which tends to $\mathord-\infty$ as $x\in \qqp$ tends to $\infty$. This leads to non-constructibility in the same way as \eqref{eqtndstmnftY} did. 

Case \eqref{propoveoDdDddb} needs an entirely different approach, which has already be given in Case 4 in pages 68--69 of Cz\'edli~\cite{czgthalesian};  take into account that our $n$ corresponds to $n+1$ in \cite{czgthalesian} and our $c$ and the $\sqrt c$ in  \cite{czgthalesian} are equivalent for constructibility.
\end{proof}

\begin{proposition}\label{propqwshrmd}
There exist \emph{rational} numbers $p$, $d_1$, and $d_2$ such that 
with the angle $\delta:=2\cdot\arccos(p)\in (0,2\pi)$, 
the $3$-fan $\dfan3{\delta,d_1,d_2,1}$
exists but it cannot be constructed from $\tuple{p=\cos(\delta/2),d_1,d_2,1}$.  Also, for every $\delta\in (0,2\pi)$ such that $\cos(\delta/2)$ is \emph{transcendental}, there are
\emph{rational} numbers $d_1$ and $d_2$ such that 
$\dfan3{\delta,d_1,d_2,1}$ exists but it cannot be constructed from $\tuple{\cos(\delta/2),d_1,d_2,1}$.
\end{proposition}

\begin{proof} Like for the special values of $\delta$ considered in Cz\'edli~\cite{czgthalesian} and   Cz\'edli and Kunos~\cite{czgkunos}, the 3-fan $\dfan3{\delta,d_1,d_2,1}$ depends continuously on its parameters. We will soon see from \eqref{eqtxtthePolynom} that the corresponding  dependence  on  $\tuple{p,d_1,d_2,1}:=\tuple{\cos(\delta/2),d_1,d_2,1}$
is polynomial, whereby it remains polynomial even after fixing some parameters and letting only the rest remain indeterminates. 
Hence, a repeated use of the Rational Parameter Theorem of Cz\'edli and Kunos~\cite[Theorem 11.1]{czgkunos} shows that 
it suffices to prove that $\dfan3{\delta,d_1,d_2,1}$ cannot be constructed  in general from its parameters $p,d_1,d_2,1$. 
Hence, we can treat $p$, $d_1$ and $d_2$ as algebraically independent numbers over $\qq$, whereby we can consider them indeterminates $w$, $y$, and $z$, respectively. 
Note that although the rest of this proof is conceptually easy and it is hopefully readable without computers, the real verification has been done by  computer algebra; reference will be given later. 

We denote the half-angles corresponding to the sides at distances $y:=d_1$, $z:=d_2$ and 1  of our 3-fan 
by $\alpha:=\alpha_1$, $\beta:=\alpha_2$ and $\gamma:=\alpha_3$, respectively. For convenience, we let $\delta'=\delta/2$. Then we have that $\alpha+\beta=\delta'-\gamma$, whereby
\begin{align*}
0&=h_1:=\cos^2(\alpha+\beta)-\cos^2(\delta'-\gamma)\cr
 &=(\cos\alpha\cos\beta -\sin\alpha\sin\beta)^2 -
  (\cos\delta'\cos\gamma +\sin\delta'\sin\gamma)^2\cr
 & =\cos^2\alpha\cos^2\beta +\sin^2\alpha\sin^2\beta -
 \cos^{2}\delta'\cos^{2}\gamma -\sin^2\delta'\sin^2\gamma - s_1,
\end{align*}
where $s_1=2\cos\alpha\cos\beta\sin\alpha\sin\beta + 2\cos\delta'\cos\gamma \sin\delta'\sin\gamma$. 
Our purpose is to get rid of the sines in $s_1$ that are raised to odd exponents. Note that neither $h_1+s_1$, nor its square has sines with odd exponents. 
Since $h_1=0$, so is 
\begin{align*}
h_2&:=(h_1+s_1)^2 - s_1^2\cr
&\phantom:=(h_1+s_1)^2 - 4\cos^2\alpha\cos^2\beta\sin^2\alpha\sin^2\beta - 4 \cos^2\delta'\cos^2\gamma \sin^2\delta'\sin^2\gamma - s_2,
\end{align*}
where $s_2=8\cos\alpha\cos\beta\sin\alpha\sin\beta \cos\delta'\cos\gamma \sin\delta'\sin\gamma$. Clearly, neither $h_2+s_2$, nor its square has sines with odd exponents. Finally, since $h_2=0$, so is
\[h_3:=(h_2+s_2)^2 - s_2^2.
\]
Now we are in the position that after expanding $h_3$,
all the sines are raised to even exponents. Hence, after substituting $1-\cos^2\alpha$, \dots, $1-\cos^2\delta'$ for $\sin^2\alpha$, \dots, $\sin^2\delta'$ in $h_3$, we obtain a quaternary polynomial $h_4$ over $\zz$ such that 
\[0=h_3=h_4(\cos\alpha,\cos\delta,\cos\gamma,\cos\delta').
\]
As we did this before, see Figure~\ref{fig1} with different notation, $\cos\alpha=d_1u=yu$, $\cos\beta=d_2u=zu$, and
$\cos\gamma=d_3u=u$, while $\cos\delta'=p=w$. Substituting these equalities into $h_4$, we obtain a nonzero quaternary polynomial $h_5$ over $\zz$ such that 
\begin{equation*}
0=h_4(\cos\alpha,\cos\delta,\cos\gamma,\cos\delta')
=h_5(y,z,w,u).
\end{equation*}
Substituting $x$ for $u^2$ in $h_5$, 
we obtain a polynomial 
\begin{equation}
\begin{aligned}
h_{y,z,w}(x)&=h(y,z,w,x) = 16 y^4z^4\cdot x^6 \cr
&+ \bigl(-16  y^2 z^6 p^2-16  y^4 z^4-16  y^6 z^2 p^2-16  y^4 z^2 -16  y^2 z^4\cr
&\phantom{+ \bigl(\kern 1pt}  -16  y^2 z^2 p^2+8  y^2 z^2+8  y^2 z^6+8  y^6 z^2\bigr) \cdot x^5 +\dots+w^8
\end{aligned}
\label{eqtxtthePolynom}
\end{equation}
over $\zz$ such that $u^2$ is a root of this polynomial 
and, as the leading coefficient indicates,  $\degx x {h({y,z,w,x})}=6$. 
For the rest of the coefficients, the reader can but need not see the
Maple worksheet to be mentioned in the proof of Theorem~ \ref{thmvltcorolyclS}\eqref{thmvltcorolyclSe}.
Since the polynomial in \eqref{eqtxtthePolynom} is too long to be fully presented here,  we display 
\begin{align*}
h(2,3,2,x)&=
20736x^6 -225792x^5 \cr
&\phantom{=} +453376x^4-180224x^3+37632x^2-3584x+256.
\end{align*}
Note that  $\degx x {h({2,3,2,x})}= \degx x {h({y,z,w,x})}$. Thus,  if $h(y,z,w,x)$ was reducible,
then so would $h({2,3,2,x})$. With the help of computer algebra, we obtain that  $h(2,3,2,x)$ is irreducible, whence so is $h(y,z,w,x)$.
Furthermore, the degree $\degx x {h(y,z,w,x)}=6$ is not a power of 2. Thus, a reference to Proposition~\ref{propwhncStr} completes the proof of Proposition~\ref{propqwshrmd}.
\end{proof}

Next, we outline another approach, which does not need computer algebra but it is conceptually harder and less detailed.

\begin{proof}[Second proof of Proposition~\textup{\ref{propqwshrmd}}]
By the Rational Parameter Theorem of Cz\'edli and Kunos~\cite{czgkunos}, it suffices to deal with the second part of  Proposition~\ref{propqwshrmd}. 
Let $\mathbb T$ denote the set of \emph{transcendental real numbers}. By \cite[Proposition 1.3]{czgkunos}, {which was taken from Cz\'edli and Szendrei~\cite{czgsza},}
there exist $d_1',d_2',d_3'\in \mathbb N$ such that 
$\dfan3{2\pi, d_1',d_2',d_3'}$ exists but it is not constructible. Clearly, with $d_1=d_1'/d_3'$, $d_2=d_2'/d_3'$, and $d_3=1=d_3'/d_3'$, the same holds for
$\dfan3{2\pi, d_1,d_2,d_3}$. By continuity,
$-1=\cos(2\pi/2)$ has  a small right neighborhood $U=(-1,-1+\varepsilon)$ such that 
 $\dfan3{2\cdot\arccos p, d_1,d_2,d_3}$ exists for every $p\in U\cap\mathbb T$. We can assume that the rational numbers $d_1$, $d_2$, and $d_3$ serve only as information; the task is to construct the $3$-fan $\dfan3{2\cdot\arccos p, d_1,d_2,d_3}$ from $p$. 
Up to isomorphism (over $\qq$), the field $\qqp$ and the constructibility problem does not depend on the choice of $p\in U\cap\mathbb T$.  Hence, either the 3-fan is non-constructible for every  $p\in U\cap\mathbb T$, or it is constructible for every  $p\in U\cap\mathbb T$. For the sake of contradiction, suppose that the second alternative holds. Then it follows by the Limit Theorem, which is Cz\'edli and Kunos~\cite[Theorem 9.1]{czgkunos}, that $\dfan3{2\pi, d_1,d_2,d_3}$ is also constructible, which contradicts the choice of $\tuple{d_1,d_2,d_3}$. Thus, the first alternative holds, and it implies  Proposition~\ref{propqwshrmd}.
\end{proof}

\begin{proof}[Proof of Theorem~\textup{\ref{thmmain}}]
Apply Propositions~\ref{propevennddd}, \ref{propositionTcDDgNgl}, \ref{propoveoDdDdd}, and  \ref{propqwshrmd}.
\end{proof}

\begin{proof}[Proof of Theorem~\textup{\ref{thmvltcorolyclS}}]
Keeping Remark~\ref{remarkdhgzVbMq} in mind, observe that  \ref{thmvltcorolyclS}\eqref{thmvltcorolyclSa} has already been proved; see   Cz\'edli and Kunos~\cite[Theorem 1.2 and Proposition 1.3]{czgkunos} together with Cz\'edli~\cite[Corollary 1.3]{czgthalesian}. Similarly, \ref{thmvltcorolyclS}\eqref{thmvltcorolyclSb} follows from \cite[Theorem 1.1]{czgthalesian}.
Propositions~\ref{propevennddd}, \ref{propositionTcDDgNgl}, \ref{propoveoDdDdd},   and \ref{propqwshrmd} imply 
\ref{thmvltcorolyclS}\eqref{thmvltcorolyclSc}.
The first inclusion in \ref{thmvltcorolyclS}\eqref{thmvltcorolyclSd} follows from Propositions~\ref{propevennddd} and \ref{propoveoDdDdd}, while the second one comes from 
Proposition~\ref{propositionTcDDgNgl}. 
Next, \ref{thmvltcorolyclS}\eqref{thmvltcorolyclSe} is a consequence of \ref{thmvltcorolyclS}\eqref{thmvltcorolyclSa} and \ref{thmvltcorolyclS}\eqref{thmvltcorolyclSd}. Finally, our proof of  \ref{thmvltcorolyclS}\eqref{thmvltcorolyclSf} needs the brute force of a computer; an appropriate program (called Maple worksheet) for Maple V, version 5, is available from the the authors' homepages. For every $w_0$ in the set $A_{1000}^{(1)}\cup A_{100}^{(2)}$, the program has to verify that  the polynomials  $h(y,z,w_0,x)$ and $h(y,z,-w_0,x)$,  see \eqref{eqtxtthePolynom}, are irreducible in $\zz[y,z,x]$. 
The program had to verify  $1{,}675{,}500$ polynomials; this took 42 minutes  with the help of a personal computer with IntelCore i5-4440 CPU, 3.10 GHz, and 8.00 GB RAM. (Note that the $1{,}675{,}500$ polynomials are not pairwise distinct; for example, each of the fractions $1/2$, $2/4$, $3/6$, \dots, and $500/1000$ gives the same $w_0$ and the same   $h(y,z,w_0,x)$.)
As the leading coefficient in \eqref{eqtxtthePolynom} shows, all these polynomials are of degree 6 with respect to $x$, independently from $w_0$. Thus, their irreducibility proves \ref{thmvltcorolyclS}\eqref{thmvltcorolyclSf}.
\end{proof}

\begin{proof}[Proof of Proposition~\textup{\ref{propsLengths}}]
The last sentence of Remark~\ref{remarkZhF} shows that 
the 3-fan
$\sfan 3{\delta, a_1,a_2,a_3}$ cannot be constructed from its central angle an side lengths in general. The same conclusion can be derived from the non-constructibility of the regular nonagon if we choose $\delta=2\pi/3$. 
Hence, the Limit Theorem from Cz\'edli and Kunos~\cite{czgkunos} implies that $\sfan n{\delta, a_1,\dots,a_n}$ is non-constructible for every $n\geq 3$. Note that the Limit Theorem applies also to a fixed central angle, whereby for every $n\geq 3$, say, $\sfan n{\pi, a_1,\dots,a_n}$ and  $\sfan n{2\pi/3, a_1,\dots,a_n}$ are non-constructible from their side lengths. The 1-fan is obviously constructible.

We are left with the case $n=2$, that is, with the constructibility of  $\sfan 2{\delta, a_1,a_2}$.
By changing the unit if necessary, we can assume that $a_1=1$.  With  $u:=1/(2r)$,  Figure~\ref{fig1} gives that $\sin(\alpha_1)= u$ and $\sin(\alpha_2)=a_2u$.
Using that $\delta':=\delta/2=\alpha_1+\alpha_2$ and denoting $\cos(\delta')$ by $p$, the binary trigonometric addition formula for cosine gives that 
$p=\cos(\alpha_1+\alpha_2)=\sqrt{1-u^2}\cdot \sqrt{1-a_2^2u^2} - u\cdot a_2\cdot u$. Substituting $x$ for $u^2$, rearranging, squaring, and rearranging again we conclude that $u^2$ is a root of the polynomial
\begin{equation}
(a_2^2 + 2pa_2 +1)x+p^2-1.
\label{eqdPshBjRb}
\end{equation}
Since $p>-1$ and $a_2$ is positive, $a_2^2 + 2pa_2 +1>
a_2^2 + 2\cdot(-1)\cdot a_2 +1 = (a_2-1)^2\geq 0$.
Hence, the coefficient of $x$ above in nonzero and 
\eqref{eqdPshBjRb} is a polynomial of degree 1. Since $u^2$ is a root of this polynomial, $u^2$ and $\sfan 2{\delta, a_1,a_2}$ are constructible.
\end{proof}

Finally, we note that although we could use the  Limit Theorem from \cite{czgkunos} to give a short approach to the constructibility of $\sfan n{\delta, a_1,\dots,a_n}$ from its central angle and side lengths and we could  apply this theorem even for the central angle in the Second proof of Proposition~\ref{propqwshrmd}, the Limit Theorem is not applicable for the central distances of our $n$-fans. This is one of the reasons that, as we know from Theorem~\ref{thmvltcorolyclS}\eqref{thmvltcorolyclSa}, there is a gap in $\noncn(2\pi)$.


\begin{thebibliography}{99}

\bibitem{cameron}
  Cameron, P. J.:
  Introduction to Algebra. 2nd edition, Oxford University Press, 2008

\bibitem{czgproblembook}
    Cz\'edli, G.:
    Problem Book on Geometric Constructibility.
    JATEPress (Szeged) 2001, 149 pages (in Hungarian)

\bibitem{czgthalesian}
  Cz\'edli, G.:
  Geometric constructibility of Thalesian polygons.%
\footnote{\red{Temporary note: freely downloadable from \url{http://www.acta.hu/} or  the author's homepage}}
  Acta Sci. Math. (Szeged) \tbf{83} (2017), 61--70.

\bibitem{czgsza}
  Cz\'edli, G., Szendrei, \'A.:
  Geometric constructibility.
  Polygon (Szeged), ix+329 pages, 1997 (in Hungarian, ISSN 1218-4071) 


\bibitem{czgkunos}
  Cz\'edli, G., Kunos, \'A.:
  Geometric constructibility of cyclic polygons and a limit theorem.%
\footnote{\red{Temporary note: freely downloadable from the same websites}}
   Acta Sci. Math. (Szeged) \tbf{81}, 643--683  (2015)

\bibitem{schreiber} 
   Schreiber, P.:
   On the existence and constructibility of  inscribed polygons.
   Beitr\"age zur Algebra und Geometrie \tbf{34},  195--199 (1993)


\bibitem{wantzel} 
  Wantzel, P. L.:
  Recherches sur les moyens de reconna{\^\i}tre si un Probl\`eme de G\'eom\'etrie peut se r\'esoudre avec le  r\`egle et le compas. 
  J.\ Math.\ Pures Appl. \tbf{2}, 366--372 (1837)

\end{thebibliography}
\end{document}